\documentclass[14pt]{extarticle}

\usepackage{fancyhdr}
\usepackage{amsmath}
\usepackage{amssymb}
\usepackage{amsfonts}
\usepackage{euscript}
\usepackage{oldgerm}
\usepackage{calligra}
\usepackage{mathrsfs}
\usepackage{hyperref}
\usepackage{url}
\usepackage{graphicx}
\usepackage{stmaryrd}
\usepackage{wasysym}
\usepackage{pifont}
\usepackage[geometry]{ifsym2}
\usepackage{times}

\renewcommand{\thefootnote}{\fnsymbol{footnote}}

\long\def\sfootnote[#1]#2{\begingroup
\def\thefootnote{\fnsymbol{footnote}}\footnote[#1]{#2}\endgroup}

\newtheorem{theorem}{Theorem}%[section]
\newtheorem{definition}[theorem]{Definition}

\newtheorem{corollary}[theorem]{Corollary}

\newtheorem{example}[theorem]{Example}

\newtheorem{remark}[theorem]{Remark}

\newenvironment{proof}{\noindent\mbox{\bf Proof.}}
{\hfill\mbox{\ding{111}}\bigskip}

\begin{document}

%%%%%%%%%%%%%%%%%%%%%%%%%%%%%%%%%%%%
%%%%%%%%%%%%%%%%%%%%%%%%%%%%%%%%%%%%
%%%%%%%%%%%%%%%%%%%%%%%%%%%%%%%%%%%%
%%%%%%%%%%%%%%%%%%%%%%%%%%%%%%%%%%%%
%%%%%%%%%%%%%%%%%%%%%%%%%%%%%%%%%%%%

%\thispagestyle{empty}

\hfill
{\bf {\large G\"odel's Incompleteness Phenomenon---Computationally}}

%\bigskip
%
%\hfill
%{\bf {\Large }}

\bigskip

\bigskip

\bigskip

\hfill {\em {Saeed Salehi}}

\smallskip

\hfill { {University of Tabriz \& IPM  (IRAN)}}

\pagestyle{fancy}
\lhead[]{}
\chead[]{}
\rhead[]{}
\lfoot[]{}
\cfoot[]{}
\rfoot[]{{{\em Philosophia Scienti{\ae}}, 18 (3), 2014, pp--qq.}}

\renewcommand{\headrulewidth}{0pt}
\renewcommand{\footrulewidth}{0pt}

\bigskip

\begin{abstract}
\noindent  We argue that
G\"odel's completeness theorem is equivalent to completability of consistent theories, and G\"odel's incompleteness theorem is equivalent to the fact that this completion is not constructive, in the sense that there are some consistent and recursively enumerable theories which cannot  be extended to any complete and consistent and recursively enumerable theory. Though any consistent and decidable theory can be extended to a complete and consistent  and decidable    theory. Thus deduction and consistency are not decidable in logic, and an analogue of Rice's Theorem holds for recursively enumerable theories: all the non-trivial properties of such theories are undecidable.

\medskip


\medskip


\medskip
\noindent {\bf Acknowledgements}. {The author is partially supported by grant  N$^{\underline{\rm o}}$~92030033 of the Institute for Research in Fundamental Sciences $(\bigcirc\hspace{-1.5ex}/\hspace{-1.1ex}\bullet\hspace{-1.1ex}/\hspace{-1.5ex}\bigcirc$ $\mathbb{I}\mathbb{P}\mathbb{M})$,  Tehran, Iran.}
\end{abstract}

\section*{Introduction}\label{intro}
The incompleteness theorem of Kurt G\"odel has been regarded as the most
significant mathematical result in the twentieth century,
and  G\"odel's completeness theorem
is a kind of the fundamental theorem of mathematical logic.
To avoid confusion between these two results, it is argued in the
literature that the completeness theorem is about the {\em semantic} completeness of first order logic,
and the incompleteness theorem is about the {\em syntactic} incompleteness of sufficiently strong
first order logical theories. In this paper we look at these two theorems
from another perspective.
We will argue that G\"odel's completeness theorem is a kind of completability theorem,
and G\"odel--Rosser's incompleteness theorem is a kind of  incompletability theorem
in a constructive manner. By G\"odel's semantic incompleteness
theorem we mean the statement that {\sl any sound and sufficiently strong and recursively enumerable theory is incomplete}.
By G\"odel--Rosser's incompleteness theorem we mean the statement that {\sl any
consistent and sufficiently strong and recursively enumerable theory is incomplete}. G\"odel's original incompleteness
theorem's assumption is between soundness and consistency; it assumes $\omega-$consistency of
sufficiently strong and recursively enumerable theories which are to be proved  incomplete.

It is noted in the literature that the existence of a non--recursive  but
recursively enumerable set can prove G\"odel's semantic incompleteness theorem (see e.g.
[Lafitte 2009]   or  [Li \& Vit\'{a}nyi 2008]).
This beautiful proof is most likely first proposed by [Kleene 1936] and Church; below we will
give an account of this proof after Theorem \ref{pr-unre}.
A clever modification of this proof
shows G\"odel--Rosser's
(stronger) incompleteness theorem, and in fact provides an elementary and
nice proof of G\"odel--Rosser's theorem other than the classical Rosser's trick ([Rosser 1936]).
This is called {\em Kleene's Symmetric Form of G\"odel's Incompleteness Theorem} (see [Beklemishev 2010]) originally published in [Kleene 1950] and later in the book [Kleene 1952].
Indeed, G\"odel's  semantic incompleteness theorem is equivalent to the existence of a non--recursive but recursively enumerable set, and also G\"odel--Rosser's (constructive) incompleteness theorem is equivalent to the existence of a pair of recursively (effectively) inseparable  recursively enumerable sets.

We will present a theory which is computability theoretic in nature, in a first order language which
 does not contain any arithmetical operations like addition or multiplication, nor set theoretic relation like membership nor
sting theoretic operation like concatenation. We will use a  ternary relation symbol $\tau$ which resembles Kleene's T predicate and our theory resembles Robinson's
\texttt{R} arithmetic (see [Tarski, et.\! al.\!  1953]). The proofs avoid using the diagonal (or fixed--point) lemma which is
highly counter--intuitive and a kind of `pulling a rabbit out of the hat'
(see [Wasserman 2008]); the proofs are also constructive,
in the sense that given a recursively enumerable theory that can interpret  our theory one can
algorithmically produce an independent sentence.  For us the simplicity
of the proofs and elementariness of the arguments are of essential importance.
Though we avoid coding sentences and proofs and other syntactic notions, coding
programs is needed for interpreting the $\tau$ relation. We also do not need any mathematical definition
for algorithms or programs (like recursive functions or Turing machines etc); all we  need
is the finiteness of programs (every program is a finite string of {\sc ascii}\footnote{\url{http://www.ascii-code.com/}} codes)
and the finiteness of input and time of computation
(which can be coded or measured by natural numbers). So, Church's Thesis (that every
intuitively computable function is a recursive function, or a function defined
rigorously in a mathematical framework)
is not used in the arguments.

%%%%%%%%%%%%%%%%%%%%%%%%%%%%%%%%%%%%
%%%%%%%%%%%%%%%%%%%%%%%%%%%%%%%%%%%%
%%%%%%%%%%%%%%%%%%%%%%%%%%%%%%%%%%%%
%%%%%%%%%%%%%%%%%%%%%%%%%%%%%%%%%%%%
%%%%%%%%%%%%%%%%%%%%%%%%%%%%%%%%%%%%
%--------------------------------------

\lhead[\thepage]{{\em  G\"odel's Incompleteness Phenomenon---Computationally}}
\chead[]{}
\rhead[{\em Saeed Salehi}]{\thepage}
\lfoot[]{\footnotesize{18(3), 2014, pp--qq.}}
\cfoot[]{}
\rfoot[\footnotesize{\sl Philosophia Scienti{\ae}}]{}

\section*{Completeness and Completability}
In mathematical logic, a {\em theory} is said to be a set of sentences,  in a fixed language (see e.g. [Chiswell \& Hodges 2007]; in [Kaye 2007] for example the word ``theory" does not appear in this sense at all, and instead ``a set of sentences" is used).  Sometimes a theory is required to be closed under (logical) deduction, i.e., a set of sentences $T$ is called a theory if for any sentence $\varphi$ which satisfies $T\vdash\varphi$ we have $\varphi\in T$ (see e.g. [Enderton 2001]). Here, by a {\em theory} we mean any set of sentences (not necessarily closed under deduction). Syntactic completeness of a theory is usually taken to be {\em negation}--completeness: a theory $T$ is {\em complete} when for any sentence $\varphi$, either $T\vdash\varphi$ or $T\vdash\neg\varphi$. Let us look at the  completeness with respect to  other connectives:

\begin{definition}[Completeness]\label{def-completeness}{\rm
A theory $T$ is called

$\bullet$  $\neg-$complete when for any sentence  $\varphi$:
\newline\centerline{$T\vdash\neg\varphi \iff T\not\vdash\varphi$.}

 $\bullet$  $\wedge-$complete when for any sentences $\varphi$ and $\psi$:
\newline\centerline{$T\vdash\varphi\wedge\psi \iff T\vdash\varphi\text{ and }T\vdash\psi$.}

$\bullet$  $\vee-$complete when for any sentences $\varphi$ and $\psi$:
\newline\centerline{$T\vdash\varphi\vee\psi \iff T\vdash\varphi\text{ or }T\vdash\psi$.}

$\bullet$  $\rightarrow-$complete when for any sentences $\varphi$ and $\psi$:
\newline\centerline{$T\vdash\varphi\rightarrow\psi \iff \text{if }T\vdash\varphi\text{ then }T\vdash\psi$.}

$\bullet$  $\forall-$complete when for every formula $\varphi(x)$:
\newline\centerline{$T\vdash\forall x\varphi(x) \iff \text{for every } t, T\vdash\varphi(t)$.}

$\bullet$  $\exists-$complete when for every formula $\varphi(x)$:

{\qquad\qquad\qquad $T\vdash\exists x\varphi(x) \iff \text{for some } t, T\vdash\varphi(t)$.\hfill\DiamondShadowB}
}\end{definition}

Let us note that the half of $\neg-$completeness is {\em consistency}:
 a theory is called {\em consistent} when for every sentence $\varphi$, if $T\vdash\neg\varphi$ then $T\not\vdash\varphi$. Usually, the other half is called {\em completeness}, i.e., when   if $T\not\vdash\varphi$ then $T\vdash\neg\varphi$ for every sentence $\varphi$.

\begin{remark}\label{rem-completeness}{
Any theory is $\wedge-$complete and $\forall-$complete (in first--order logic).
Also, half of $\vee,\rightarrow,\exists-$completeness holds for all theories $T$; i.e.,

--- if $T\vdash\varphi$ or $T\vdash\psi$, then $T\vdash\varphi\vee\psi$;

--- if $T\vdash\varphi\rightarrow\psi$, then if $T\vdash\varphi$ then $T\vdash\psi$;

--- if $T\vdash\varphi(t)$ for some $t$, then $T\vdash\exists x\varphi(x)$.
}\hfill\DiamondShadowB\end{remark}

A {\em maximally consistent} theory is a theory $T$ which cannot properly be extended to a consistent theory; i.e., for any consistent theory $T'$ which satisfies $T\subseteq T'$ we have $T=T'$. The following is a classical result in mathematical logic (see e.g. [van Dalen 2013]).

\begin{remark}\label{rem-completable}A consistent theory is $\neg-$complete if and only if  is $\vee-$complete if and only if is $\rightarrow-$complete if and only if is maximally consistent.
\hfill\DiamondShadowB\end{remark}

Consistently maximizing a theory suggests using Zorn's Lemma or (equivalently) the Axiom of Choice, which is non--constructive in general. To see if one can do it constructively or not, we need to introduce some other notions. Before that let us note that $\exists-$completing a theory can be done constructively.

\begin{remark}\label{rem-ecompletable}
Any arbitrary first--order  consistent theory can be extended (constructively) to another consistent $\exists-$complete theory.
\hfill\DiamondShadowB\end{remark}

The main idea of the proof is that we add a countable set of constants $\{c_1,c_2,\cdots\}$
to the language, and then enumerate all the couples of formulas and variables
in the extended language as $\langle\varphi_1,x_1\rangle, \langle\varphi_2,x_2\rangle,\cdots$ and finally add the sentences $\exists x_1\varphi_1\rightarrow\varphi(c_{l_1}/x_1)$, $\exists x_2\varphi_2\rightarrow\varphi(c_{l_2}/x_2)$, $\cdots$ successively to the theory, where in each step $c_{l_i}$ is the first constant which does not appear in $\varphi_1,\ldots,\varphi_i$ and has not been used in earlier steps (see e.g. [Enderton 2001]).

Let us note that $\exists-$complete theories are sometimes called Henkin theories or Henkin--complete  or Henkin sets (see e.g. [van Dalen 2013]). These are used for proving G\"odel's Completeness Theorem by Henkin's proof.
The theory of a structure is the set of  sentences (in the language of that structure) which are true in that structure. It can be seen that theories of structures are $(\neg,\exists)-$complete theories. Conversely, for any $(\neg,\exists)-$complete theory $T$ one can construct a structure $\mathcal{M}$ such that $T$ is the theory of $\mathcal{M}$.

\begin{remark}\label{rem-allcomplete}
Any consistent theory can be extended to a  $(\neg,\exists)-$complete theory. Note that any $(\neg,\exists)-$complete theory is complete with respect to all the other connectives.
\hfill\DiamondShadowB\end{remark}

G\"odel's completeness theorem is usually proved by  showing that any consistent theory has a model (the model existence theorem---which is equivalent to the original completeness theorem). Note that for proving the model existence theorem it is shown that any consistent theory is extendible to a consistent $(\neg,\exists)-$complete theory, which then defines a structure which is a model of that theory. Thus, we can rephrase this theorem equivalently  as follows.

\bigskip

\noindent {\sc G\"odel's Completeness Theorem}: {\sl Any first--order consistent theory can be  extended to a consistent $(\neg,\exists)-$complete theory.}

\bigskip

This theorem can be considered as {\em the fundamental theorem of logic},
the same way that we have the fundamental theorem of arithmetic,
or the fundamental theorem of algebra, or the fundamental theorem of calculus.
We could also call this theorem, G\"odel's {\sl Completability} Theorem, for the above reasons.

\section*{Incompleteness and Incompletability}
We now turn our attention to  constructive aspects of the above theorem.
A possibly infinite set  can be constructive  when it is {\em decidable} or (at least)
 {\em recursively enumerable}.  A set $D$ is decidable when there exists a
  single--input algorithm which on any input $x$ outputs {\tt Yes} if $x\in D$ and
  outputs {\tt No} if $x\not\in D$. A set $R$ is called recursively enumerable
  ({\sc re} for short) when there exists  an input--free algorithm which outputs
  (generates) the elements of $R$ (after running).
  It is a classical result in Computability Theory that there exists an {\sc re}
  set which is not decidable (though, any decidable set is {\sc re}); see .e.g [Epstein \& Carnielli 2008].
   For a theory $T$ we can consider decidability or recursive enumerability of either  $T$ as a set of sentences, or the set of derivable sentences of $T$, i.e., ${\rm Der}(T)=\{\varphi\mid T\vdash\varphi\}$.
It can be shown that if $T$ is decidable or {\sc re} (as a set) then
${\rm Der}(T)$ is {\sc re}; of course when ${\rm Der}(T)$ is {\sc re}
then $T$ is {\sc re} as well, and by Craig's trick ([Craig 1953]) for
such a theory there exists a decidable set of sentences $\widehat{T}$ such that
${\rm Der}(T)={\rm Der}(\widehat{T})$. So, we consider {\sc re} theories only, and call  theory $T$ a {\em decidable theory} when ${\rm Der}(T)$ is a decidable set (of sentences). {\sc re} theories are sometimes called axiomatizable theories (in e.g. [Enderton 2001]).
Below we will show that there exists some decidable set of sentences ($\mathcal{T}$)  whose set of derivable sentences  is not decidable (though it must be {\sc re} of course).

\begin{definition}[{\sc re}, Decidable and {\sc re}--Completable]\label{def-redec}{\rm
A consistent theory

$\bullet\,\,\,T$ is called  an {\sl {\sc re} theory} when ${\rm Der}(T)$ is an {\sc re} set.

$\bullet\,\,\,T$ is called  a {\sl decidable} theory when ${\rm Der}(T)$ is a decidable set.

$\bullet\,\,\,T$ is called  {\sl {\sc re}--completable}   when there exists a  theory $T'$ extending $T$ (i.e., $T\subseteq T'$) such that  $T'$ is consistent, complete and {\sc re}.
}\hfill\DiamondShadowB\end{definition}

It is a classical fact that  complete {\sc re} theories are decidable
(see e.g. [Enderton 2001]): since by recursive enumerability of $T$ both $\{\varphi\mid T\vdash\varphi\}$ and $\{\varphi\mid T\vdash\neg\varphi\}$ are {\sc re} and by the completeness of $T$ we have $\{\varphi\mid T\not\vdash\varphi\}=\{\varphi\mid T\vdash\neg\varphi\}$, so the set ${\rm Der}(T)$ and its complement are both {\sc re}
and hence decidable (by Kleene's Complementation Theorem -- see e.g. [Berto 2009]).
Completeness is a logician's tool for decidability. Henkin's completion shows
that any {\sc re} decidable theory is {\sc re}--completable (see [Tarski, et.\! al.\!  1953]).
The main idea is that having a decidable theory $T$ we list all
the sentences in the language of $T$ as $\varphi_1,\varphi_2,\cdots$ and
then add $\varphi_i$ or $\neg\varphi_i$  in the $i$th step to $T$ as follows:
let $T_0=T$ and if $T_j$ is defined let $T_{j+1}=T_j\cup\{\varphi_j\}$
if $T_j\cup\{\varphi_j\}$ is consistent, otherwise let
$T_{j+1}=T_j\cup\{\neg\varphi_j\}$. Note that if $T_j$ is consistent,
then $T_{j+1}$ will be consistent as well (if $T_j\cup\{\varphi_j\}$
is inconsistent then $T_j\cup\{\neg\varphi_j\}$ must be consistent).
The theory $T'=\bigcup_{i\geqslant 0} T_i$ will be consistent and complete.
This was essentially Henkin's Construction for proving G\"odel's completeness theorem.
 The point is that if $T$ is decidable then so is any $T_i$ since they are
 finite extensions of $T$. Finally, $T'$ is a decidable theory, because for any given sentence $\varphi$ it should appear in the list $\varphi_1,\varphi_2,\cdots$, so say $\varphi=\varphi_n$. Now, for $i=1,2,\ldots,n$ we can decide whether $\varphi_i\in T'$ or $\neg\varphi_i\in T'$ inductively; and finally we can decide whether $T'\vdash\varphi$ or not
($T\not\vdash\varphi$ happens only when $T'\vdash\neg\varphi$).
So, for consistent {\sc re} theories we have the following inclusions:

\begin{center}
\begin{tabular}{| c c c c c c c |}
\hline
& & & & &  &\\
& {\rm Complete} & $\Longrightarrow$ & {\rm Decidable} & $\Longrightarrow$ & {\sc re}--{\rm Completable} & \\
& & & & & & \\
\hline
\end{tabular}
\end{center}

Below we will see that the converse conclusions do not hold (Remark~\ref{remarkt}).
Whence, by contrapositing  the above conclusions we will have the following inclusions and non--inclusions for some consistent {\sc re} theories:

\begin{center}
\begin{tabular}{| c c c c c c c |}
\hline
& & & & & & \\
&  & $\Longrightarrow$ &   & $\Longrightarrow$ &  & \\[-0.75em]
& {\sc re}--{\rm Incompletable} &   & {\rm Undecidable} &   & {\rm Incomplete} & \\[-0.75em]
&  & $\not\Longleftarrow$ &   & $\not\Longleftarrow$  &  & \\
& & & & & & \\
  \hline
\end{tabular}
\end{center}

\noindent Incomplete theories abound in mathematics: every theory which has finite models but does not fix the number of elements (e.g. theory of groups, rings, fields, lattices, etc.) is an incomplete theory. By encoding Turing machines into a first order language one can obtain an undecidable theory (see e.g. [Boolos, et.\! al.\!  2007]). But demonstrating an {\sc re}--incompletable theory is a difficult task and it is in fact G\"odel's Incompleteness Theorem. {\sc re}--incompletable theories are know as ``essentially undecidable" theories in the literature (starting from [Tarski, et.\! al.\!  1953]).
Comparing G\"odel's Completeness Theorem with his Incompleteness Theorem,
we come to the following conclusion.

\noindent {\sl  Every consistent theory can be extended to a   $(\neg,\exists)-$complete
 theory {\rm (G\"odel's Completeness Theorem)} and this completion preserves
 decidability, i.e., every consistent  and decidable theory can be extended to a consistent, decidable
 and $(\neg,\exists)-$complete
 theory. But this completion cannot be necessarily effective; i.e., there are some consistent {\sc re} theories whose all consistent completions are non--{\sc re} {\rm (G\"odel's Incompleteness Theorem)}}.

So, calling the completeness theorem of G\"odel {\sl Completability Theorem} we can call (the first) incompleteness theorem of G\"odel (and Rosser) {\sl {\sc re}--Incompletabiliy Theorem}.

\subsection*{An Undecidable but RE--Completable Theory}

In this paper we introduce an incomplete but {\sc re}--completable theory ($\mathcal{T}$) and a novel   {\sc re}--incompletable theory ($\mathcal{S}$),
and for that we consider the theory of zero, successor and order in the set of natural numbers, i.e., the structure $\langle\mathbb{N},0,s,<\rangle$ in which $0$ is a constant symbol, $s$ is a unary function symbol and $<$ is a binary relation symbol (interpreted as the zero element, the successor function and the order relation, respectively). This theory is known to be decidable ([Enderton 2001]), and in fact can be finitely axiomatized as follows

$A_1:  \forall x\forall y (x<y\rightarrow y\not<x)$,

$A_2:  \forall x\forall y\forall z(x<y\wedge y<z\rightarrow x<z)$,

$A_3: \forall x\forall y (x<y\vee x=y\vee y<x)$,

$A_4:  \forall x\forall y (x<y\longleftrightarrow s(x)<y\vee s(x)=y)$,

$A_5:  \forall x (x\not<0)$,

$A_6:  \forall x\big(0<x\rightarrow\exists v (x=s(v))\big)$.

 \noindent The axioms $A_1,A_2,A_3$ state that $<$ is a (linear and transitive and antisymmetric, thus a) total ordering, $A_4$ states that every element has a successor (the successor $s(x)$ of $x$ satisfies
{$\forall y \big(x<y\leftrightarrow s(x)<y\vee s(x)=y\big)$}), $A_5$ states that there exists a least element (namely $0$) and finally $A_6$ states that every non--zero element has a predecessor.
  One other advantage of the language $\{0,s,<\}$ is that we have terms for every natural number $n\in\mathbb{N}$:
  \newline\centerline{$\underline{n}$ is the $\{0,s\}-$term $s^n0=\underbrace{s\cdots s}_{n-\text{times}}0$.}
 To the language $\{0,s,<\}$ we add a ternary relation symbol $\tau$  interpreted as:

\begin{center}for $e,x,t\in\mathbb{N}$ the relation $\tau(e,x,t)$ holds when
\\
$e$ is a code for a single--input program which  halts on input $x$ by time $t$.
\end{center}

\noindent
Timing of  a program can be measured either by the number of steps that the
program runs or just by the conventional seconds, minutes, hours, etc. and  programs
(say in a fixed programming language like \textbf{C++}) can be coded by natural numbers as follows (for
example):
Any such program  is a (long) string  of {\sc ascii} codes,
and every {\sc ascii} code can be thought of as  8 symbols of \textsf{0}'s and \textsf{1}'s
(so, there are 256 {\sc ascii} codes). So, any program is a string of \textsf{0}'s and \textsf{1}'s
(whose length is a multiple of 8). The set of \textsf{0,1}--strings can be coded by natural
numbers in the following way:

\medskip
\begin{tabular}{c|c|c|c|c|c|c|c|c|c|c|c|c}
$\lambda$ & \textsf{0} & \textsf{1} & \textsf{00} & \textsf{01} & \textsf{10} & \textsf{11} & \textsf{000} & \textsf{001} &  \textsf{010} & \textsf{011} & \textsf{100} & $\cdots$ \\
\hline
0 & 1 & 2 & 3 & 4 & 5 & 6 & 7 & 8 & 9 & 10 & 11 & $\cdots$  \\
\end{tabular}
\medskip

\noindent This coding works as follows: given a string of \textsf{0}'s and \textsf{1}'s (take for example \textsf{0110}), put a \textsf{1} at the beginning of it (in our example \textsf{10110}) and compute its binary value  (in our example 2+2$^2$+2$^4$=22) and subtract 1 from it (in our example 21) to get the natural number which is the code of the original string. Conversely, given a natural number (for example 29) find the binary representation of its successor (in our example 30=2+2$^2$+2$^3$+2$^4$=$(\textsf{11110})_2$) and remove the \textsf{1} from its beginning (in our example \textsf{1110}) to get the \textsf{0,1}--string which corresponds to the given natural number.

Whence, any program can be coded by
a natural number constructively, and if a natural number is a code for a program, then
that program can be decoded from that number algorithmically.
Let us note the ternary relation $\tau$ resembles Kleene's T Predicate (see [Kleene 1936]).

\begin{definition}[The Theory $\mathcal{T}$]\label{def-ttheory}{\rm
Theory $\mathcal{T}$  is axiomatized by $A_1$, $A_2$, $A_3$, $A_4$, $A_5$ and $A_6$ with the following set of sentences in the language $\{0,s,<,\tau\}$:

$\,\,\,A_7: \{\tau(\underline{e},\underline{x},\underline{t})\mid e,x,t\in\mathbb{N} \,\;\&\,\;  \mathbb{N}\models\tau(e,x,t)\}$\hfill\DiamondShadowB
}\end{definition}

The set of axioms $A_7$ consists of the sentences
$\tau(\underline{e},\underline{x},\underline{t})$
(recall that $\underline{n}$ is the $\{0,s\}-$term representing the
number $n\in\mathbb{N}$) such that $\tau(e,x,t)$ holds in reality
(the single--input program with code $e$ halts on input $x$ by time $t$).

\begin{remark}[{\sc re}--completability of $\mathcal{T}$]\label{rem-tisrecom}
The set of sentences in $\mathcal{T}$ is decidable
(given any $n,m,k$ one can decide whether $\tau(n,m,k)$ holds or not),
and thus $\mathcal{T}$ is an {\sc re} theory. It is also {\sc re}--completable,
since its extension by the sentence $\forall x\forall y\forall z\big(\tau(x,y,z)\big)$ is a decidable theory {\rm (}equivalent to the theory of the structure $\langle\mathbb{N},0,s,<\rangle$ which is decidable -- see [Enderton 2001]{\rm )}.
\hfill\DiamondShadowB\end{remark}

The theory $\mathcal{T}$ is undecidable, since the halting problem is undecidable
(see e.g. [Epstein \& Carnielli 2008]):
for any (single--input program with code) $e\in\mathbb{N}$
and any (input) $m\in\mathbb{N}$,
let $\varphi_{e,m}$ be the sentence $\exists z\, \tau(\underline{e},\underline{m},z)$. Then

\begin{tabular}{lcl}
$\mathcal{T}\vdash\varphi_{e,m}$ & $\iff$ & $\mathbb{N}\models \tau(\underline{e},\underline{m},\underline{t}) \text{ for some } t$\\
 & $\iff$ & the program $e$ eventually halts on input $m$.
\end{tabular}

\noindent This can be shown directly, by incorporating the proof of the
undecidability of the halting problem.

\begin{theorem}[Undecidability of $\mathcal{T}$]\label{th-unt}
The theory $\mathcal{T}$ is undecidable.
\end{theorem}
\begin{proof}
If the set ${\rm Der}(\mathcal{T})$ is decidable, then so is the set
\newline\centerline{$\mathcal{D}=\{n\in\mathbb{N}\mid\mathcal{T}\not\vdash\exists z\,\tau(\underline{n},\underline{n},z)\}$.} Whence, there exists a program which on input $n\in\mathbb{N}$ halts whenever $n\in\mathcal{D}$ (and when $n\not\in\mathcal{D}$ then the program does not halt and loops forever). Let $e$ be a code for this (single--input) program. Then

the program (with code) $e$ halts on input
$e$ $\iff$ $\mathbb{N}\models\tau(\underline{e},\underline{e},\underline{k})\text{ for some }k$ $\iff$
$\mathcal{T}\vdash\exists z\,\tau(\underline{e},\underline{e},z)$ $\iff$
$e\not\in\mathcal{D}\iff$ (by $e$'s definition) the program (with code) $e$ does not halt on input $e$. \qquad\qquad Contradiction!
\end{proof}

\begin{remark}\label{remarkt}
Thus far, we have shown that an undecidable theory need not be {\sc re}-incompletable ($\mathcal{T}$). One can also show that an incomplete theory need not be undecidable; to see this consider the theory $\{\exists x\exists y\forall z (z=x\vee z=y)\}$ in the language of equality $(=)$. This theory is decidable (holds in models of at most two elements) but not complete, since can derive neither $\forall x\forall y (x=y)$ nor $\neg\forall x\forall y (x=y)$.
\end{remark}

\begin{corollary}[Undecidability of Consistency]\label{cor-con}
 It is not decidable whether a given {\sc re} theory is consistent or not.
\end{corollary}
\begin{proof}
By $[T\vdash\varphi]\!\!\iff\!\![T\cup\{\neg\varphi\}$ is inconsistent$]$ if consistency
of {\sc re} theories was decidable then every {\sc re} theory would be decidable too.
\end{proof}

Thus, we have shown the existence of an {\sc re} theory ($\mathcal{T}$) which is undecidable but {\sc re}--completable. Next, we show the existence of an {\sc re} theory which is not {\sc re}--completable. Before that let us note that the above proof works for any (consistent and {\sc re}) theory $T\supseteq\mathcal{T}$ which is sound (i.e., $\mathbb{N}\models T$).

\begin{theorem}\label{pr-unre}
There exists no  complete, sound  and {\sc re} theory extending $\mathcal{T}$.  In other words, the theory $\mathcal{T}$ cannot be soundly {\sc re}--completed. \hfill\ding{111}
\end{theorem}

This is essentially the semantic form of G\"odel's first incompleteness theorem.
As a corollary we have that the theory of
 $\langle\mathbb{N},0,s,<,\tau\rangle$ is not {\sc re} (nor decidable).
 Let us note that the above proof of the first (semantic) incompleteness theorem
 of G\"odel is some
 rephrasing of Kleene's proof (see [Kleene 1936]).
  For a (single--input) program (with code) $e$ let $W_e$ be the set of the inputs
  such that $e$ eventually halts on them, i.e., $W_e\!=\!\{n\mid\mathbb{N}\models\exists z\,\tau(e,n,z)\}$.
  By Turing's results it is known that the set $K=\{n\mid n\in W_n\}$ is {\sc re}
  but not decidable (see e.g. [Epstein \& Carnielli 2008]).
  Indeed, its complement $\overline{K}\!=\!\{n\mid n\not\in W_n\}$ is not {\sc re}
  because every {\sc re} set is of the form $W_m$ for some $m$ (see e.g. [Epstein \& Carnielli 2008]) and for any $n$ we have $n\in (\overline{K}\setminus W_n)\cup (W_n\setminus\overline{K})$.
  On the other hand for any {\sc re} theory $T$ the  set
  \newline\centerline{$\overline{K}_T\!=\!\{n\mid T\vdash ``n\in \overline{K}"\}\!=\!\{n\mid T\vdash\neg\exists z\,\tau(\underline{n},\underline{n},z)\}$} is {\sc re}.
    Now, if $T$ is sound ($\mathbb{N}\models T$) then
    $\overline{K}_T\subseteq \overline{K}$. The inclusion must be
    proper because one of them ($\overline{K}_T$) is {\sc re} and the other one ($\overline{K}$)
    is not {\sc re}. If  $\overline{K}_T=W_m$ (for some $m\in\mathbb{N}$) then  $m\in \overline{K} - \overline{K}_T$: because if $m\in\overline{K}_T$($=W_m$) then $m\in W_m$ and so $m\not\in\overline{K}$, and this contradicts the inclusion $\overline{K}_T\subseteq\overline{K}$; thus $m\not\in\overline{K}_T$ and so $m\not\in W_m$ which implies that $m\in\overline{K}$. Hence, the sentence $\neg\exists z\,\tau(\underline{m},\underline{m},z)$ is true but unprovable in $T$; thus $T$ is incomplete.

\subsection*{An RE--Incompletable Theory}

In the above arguments we used the soundness assumption of $T$ (and $\mathcal{T}$). Below, we will introduce a consistent and {\sc re} theory $\mathcal{S}$ which is not {\sc re}--completable.  Let $\pi$ be a binary function symbol (representing some pairing -- for example $\pi(n,m)=(n+m)^2+n$) whose interpretation in $\mathbb{N}$ satisfies the pairing condition: for any $a,b,a',b'\in\mathbb{N}$ we have $\mathbb{N}\models\pi(a,b)=\pi(a',b')$ if and only if $a=a'$ and $b=b'$.

\begin{definition}[The Theory $\mathcal{S}$]\label{def-theorys}{\rm
The theory $\mathcal{S}$ is the extension of the theory $\mathcal{T}$  by the following sets of sentences in the language $\{0,s,<,\tau,\pi\}$:

 $\,\,\,A_8: \{\neg\tau(\underline{e},\underline{x},\underline{t})\mid e,x,t\in\mathbb{N} \,\;\&\,\;  \mathbb{N}\not\models\tau(e,x,t)\}$

  $\,\,\,A_9: \{\forall x\big(x<\underline{k}\longleftrightarrow\bigvee_{i<k}x=\underline{i}\big)\mid k\in\mathbb{N}\}$
 }\hfill\DiamondShadowB\end{definition}

We now show that the theory $\mathcal{S}$ is not {\sc re}--completable. Let us note that G\"odel's original first incompleteness theorem showed the existence of some theory which was not {\em soundly} {\sc re}--completable. Actually, G\"odel used a syntactic notion weaker than ``soundness", namely {\em $\omega-$consistency}, which is stronger than consistency itself. Nowadays it is known that G\"odel's proof works for an even weaker condition than $\omega-$consistency, the so called {\em 1--consistency} (see [Isaacson 2011]).
It was then Rosser who showed that G\"odel's theorem can be proved without using the $\omega-$consistency (even 1-consistency) assumption (see [Rosser 1936] or e.g. [Boolos, et.\! al.\! 2007]); so the theorem of G\"odel--Rosser states the existence of a consistent and {\sc re} theory which is not {\sc re}--completable.

\begin{theorem}[{\sc re}--incompletability of $\mathcal{S}$]\label{th-sreincom}
If $T$ is a consistent {\sc re} theory that extends $\mathcal{S}$ (i.e., $T\supseteq\mathcal{S}$), then $T$ is not complete.
\end{theorem}

\begin{proof}
Suppose $T$ is a consistent and {\sc re} extension of $\mathcal{S}$. We show that $T$ is not complete. For any $a,b\in\mathbb{N}$ let $\varphi_{a,b}$ be the sentence
\newline\centerline{$\exists x\big(\tau(\underline{a},\pi(\underline{a},\underline{b}),x)\wedge\forall y\!<\!x\,\neg\tau(\underline{b},\pi(\underline{a},\underline{b}),y)\big)$.}
Let $m$ be a code of a program which on input $p\in\mathbb{N}$ halts
if and only if there are some $k,l\in\mathbb{N}$
such that $p=\pi(k,l)$ (in which case the numbers $k$ and $l$ are unique)
and there exists a proof of $\varphi_{k,l}$ in $T$
(i.e., $T\vdash\varphi_{k,l}$). So, if (i) $p$ is not in the range of the
function $\pi$, or (ii)  there are (unique) $k,l$ such that $p=\pi(k,l)$
and $T\not\vdash\varphi_{k,l}$, then the program does not halt on $p$. Whence, the program with code $m$
searches for a proof of $\varphi_{k,l}$ in $T$ on  input $\pi(k,l)$.

Also, let $n$ be a code for a program which for an input
$p\in\mathbb{N}$ halts if and only if there are some (unique)
$k,l\in\mathbb{N}$ such that $p=\pi(k,l)$  and there exists a proof of
$\neg\varphi_{k,l}$ in $T$ (i.e., $T\vdash\neg\varphi_{k,l}$).
So, if $p$ is not in the range of the function $\pi$ or if there are (unique) $k,l$
such that $p=\pi(k,l)$ and $T\not\vdash\neg\varphi_{k,l}$ then the program with
code $n$ does not halt on $p$. Again, this program searches for a proof of
$\neg\varphi_{k,l}$ in $T$ on input $\pi(k,l)$.

We prove that $\varphi_{n,m}$ is independent from $T$, i.e., $T\not\vdash\varphi_{n,m}$ and $T\not\vdash\neg\varphi_{n,m}$.

(1) If $T\vdash\varphi_{n,m}$ then by the consistency of $T$ we have
$T\not\vdash\neg\varphi_{n,m}$. So, on input $\pi(n,m)$ the program with code
$m$ halts  and the program with code $n$ does not halt. Whence, for some natural
number $t$, $\mathbb{N}\models\tau(m,\pi(n,m),t)$ and for every natural
number $s$, $\mathbb{N}\not\models\tau(n,\pi(n,m),s)$. So by $A_7$ and $A_8$ for that (fixed)
 $t\in\mathbb{N}$ we have
 $T\vdash\tau(\underline{m},\pi(\underline{n},\underline{m}),\underline{t})$,
 and for every $s\in\mathbb{N}$ we have
 $T\vdash\neg\tau(\underline{n},\pi(\underline{n},\underline{m}),\underline{s})$.
 Thus,
 $T\vdash\bigwedge_{i\leqslant t}\!\neg\tau(\underline{n},\pi(\underline{n},\underline{m}),\underline{i})$,
 and so by $A_9$, we conclude that
  $T\vdash\forall x\!\leqslant\!\underline{t}\,\neg\tau(\underline{n},\pi(\underline{n},\underline{m}),x)$,
  therefore

  (i) $T\vdash\forall x\big(\tau(\underline{n},\pi(\underline{n},\underline{m}),x)\!\rightarrow\!x\!>\!\underline{t}\big)$.

\noindent   Also by $T\vdash\tau(\underline{m},\pi(\underline{n},\underline{m}),\underline{t})$
  we get

  (ii) $T\vdash\forall x\!>\!\underline{t}\big(\exists y\!<\!x\,\tau(\underline{m},\pi(\underline{n},\underline{m}),y)\big)$.

  \noindent
  Combining these two conclusions we infer that
  \newline\centerline{$T\vdash\forall x\,\big(\tau(\underline{n},\pi(\underline{n},\underline{m}),x)\rightarrow \exists y\!<\!x\,\tau(\underline{m},\pi(\underline{n},\underline{m}),y)\big)$.}
  On the other hand by the definition of $\varphi_{a,b}$ we have
   \newline\centerline{$\varphi_{n,m}\equiv \exists x\big(\tau(\underline{n},\pi(\underline{n},\underline{m}),x)\wedge\forall y\!<\!x\,\neg\tau(\underline{m},\pi(\underline{n},\underline{m}),y)\big)$,} and so $\neg\varphi_{n,m}\equiv\forall x\big(\tau(\underline{n},\pi(\underline{n},\underline{m}),x)\rightarrow\exists y\!<\!x\,\tau(\underline{m},\pi(\underline{n},\underline{m}),y)\big)$.

\noindent Thus we deduced $T\vdash\neg\varphi_{n,m}$ from the assumption $T\vdash\varphi_{n,m}$;
contradiction! Whence, $T\not\vdash\varphi_{n,m}$.

(2) If $T\vdash\neg\varphi_{n,m}$ then (again) by the consistency of $T$ we have
$T\not\vdash\varphi_{n,m}$. So, on input $\pi(n,m)$ the program with code $n$ halts
and the program with code $m$ does not halt. Whence, for some natural number $t$,  $\mathbb{N}\models\tau(n,\pi(n,m),t)$ and for every natural number
$s$, $\mathbb{N}\not\models\tau(m,\pi(n,m),s)$. Similarly to the above we
can conclude that
$T\vdash\tau(\underline{n},\pi(\underline{n},\underline{m}),\underline{t})$
 and
 $T\vdash\forall y\!<\!\underline{t}\,\neg\tau(\underline{m},\pi(\underline{n},\underline{m}),y)$.
 Thus (for $x=\underline{t}$) we have
   \newline\centerline{$T\vdash\exists x\big(\tau(\underline{n},\pi(\underline{n},\underline{m}),x)\wedge\forall y\!<\!x\,\neg\tau(\underline{m},\pi(\underline{n},\underline{m}),y)\big)$} or $T\vdash\varphi_{n,m}$;
 contradiction! So, $T\not\vdash\neg\varphi_{n,m}$.

\noindent Whence, $T$ is not complete.
\end{proof}

The above proof is effective, in the sense that given an {\sc re}
 theory (by a code for a program that generates its elements) that
 extends $\mathcal{S}$ one can generate (algorithmically) a sentence
  which is independent from that theory. Let us note that for proving
  {\sc re}--incompletability of theories, it suffices to interpret $\mathcal{S}$
  in them. So, the theories $\texttt{Q}$, $\texttt{R}$ (see [Tarski, et.\! al.\! 1953])
   and Peano's Arithmetic $\texttt{PA}$ are all {\sc re}--incompletable (or, essentially undecidable).

\section*{Rice's Theorem for RE Theories}
In this last section
 we show a variant of Rice's Theorem for logical theories.  In [Oliveria \& Carnielli 2008] the authors (claimed to) had shown that an analogue of Rice's theorem holds for finitely axiomatizable first order theories. Unfortunately, the result was too beautiful to be true ([Oliveria \& Carnielli 2009]) and it turned out that Rice's theorem cannot hold for finite theories. However, we show that this theorem holds for {\sc re} theories, a result which is not too different from Rice's original theorem.
Recall that two theories $T_1$ and $T_2$ are equivalent when they prove the same (and exactly the same) sentences (i.e., ${\rm Der}(T_1)={\rm Der}(T_2)$).

\begin{definition}[Property of Theories]\label{def-property}{\rm
A property of (first order logical) theories is a set of natural numbers
$\mathcal{P}\subseteq\mathbb{N}$ such that for any $m,n\in\mathbb{N}$ if the
theory generated by the program with code $m$ is equivalent to the theory
generated by the program with code $n$, then $m\in\mathcal{P}\longleftrightarrow n\in\mathcal{P}$.

So, a theory is said to have the property $\mathcal{P}$ when a code for generating its set belongs to $\mathcal{P}$.
A property of theories is a non--trivial property when some theories have that property and some do not.
 }\hfill\DiamondShadowB\end{definition}

\begin{example}\label{ex-property}{\rm
The followings are some non--trivial properties of {\sc re} theories:

$\bullet\,$ {\sl Universal Axiomatizability}: theories axiomatizable by sentences of the

\quad form $\forall x_1\ldots\forall x_n\theta(x_1,\ldots,x_n)$ for quantifier--free $\theta$'s;

$\bullet\,$ {\sl Finite Axiomatizability}: being equivalent to a finite theory;

$\bullet\,$ {\sl Decidability} (of the set of the theorems of  the theory);

%$\bullet\,$ Deriving a Certain Sentence (that is not a tautology or a contradiction);

$\bullet\,$ {\sl Having a Finite Model};

$\bullet\,$ {\sl Completeness};

$\bullet\,$ {\sl Consistency}.
}\hfill\DiamondShadowB\end{example}

\begin{remark}\label{rem-property}
For any non--trivial property $\mathcal{P}$ either (i) no inconsistent theory
has the property $\mathcal{P}$ or (ii) all inconsistent theories have
the property $\mathcal{P}$. Because when an inconsistent theory breaks into $\mathcal{P}$
then all the other inconsistent theories (being equivalent to each other) come in.
\hfill\DiamondShadowB\end{remark}

Before proving Rice's Theorem let us have a look at (a variant of) Craig's trick.
For an {\sc re} theory $T=\{T_1,T_2,T_3,\cdots\}$ the proof predicate ``$p$ is a proof
of $\varphi$ in $T$", for given sequence of sentences $p$ and sentence $\varphi$,
might not
be decidable when the set $\{T_1,T_2,T_3,\cdots\}$ is not decidable. Note that ``the sequence
$p$ is a proof of $\varphi$ in $T$" when every element of $p$ is either a (first order)
logical axiom (which can be decided) or is an element of $T$ or can be deduced from
two previous elements by an inference rule, and the last element of $p$ is $\varphi$.
Thus decidability of the set $T$ is essential
for the decidability of the proof sequences of $T$.
But if we consider the theory  $\widehat{T}=\{\widehat{T}_1,\widehat{T}_2,\widehat{T}_3,\cdots\}$
where $\widehat{T}_m=\bigwedge_{i\leqslant m}T_i$ then the set $\widehat{T}$ is decidable, because
if $\mathcal{A}$ is an algorithm that outputs (generates)  the infinite sequence $\langle T_1,T_2,T_3,\cdots\rangle$ in this order (in case $T$ is finite the sequence is eventually constant), then for any given sentence $\psi$ we can decide if $\psi\in\widehat{T}$ or not
by checking if $\psi$ is a conjunction of some sentences
$\psi=\psi_1\wedge\cdots\wedge\psi_m$ (if $\psi$ is not of this form, then already
$\psi\not\in\widehat{T}$)
 such that $\mathcal{A}$'s $i$th output is $\psi_i$ for $i=1,\ldots,m$
 (if not then again $\psi\not\in\widehat{T}$). Whence the predicate of
 being a proof of $\varphi$ in $\widehat{T}$,
i.e.,
  ``the sequence
 $p$ is a proof of sentence $\varphi$ in $\widehat{T}$ ", is decidable; moreover
  the theories $T$ and $\widehat{T}$ are equivalent, and the theory $\widehat{T}$ can be algorithmically constructed from   given theory $T$.

\begin{theorem}[Analogue of Rice's Theorem]\label{th-rice}
All the non--trivial properties of {\sc re} theories are undecidable.
\end{theorem}

\begin{proof}
Assume a non--trivial property $\mathcal{P}$ of {\sc re} theories is decidable, i.e.,
there exists an algorithm which on input $n\in\mathbb{N}$ decides
whether $n\in\mathcal{P}$ (i.e., whether the theory generated by the program with
code $n$ has the property $\mathcal{P}$). Without loss of
generality we can assume that no inconsistent theory has the
property $\mathcal{P}$ (otherwise take the complement of $\mathcal{P}$).
Fix a consistent {\sc re} theory, say,   $S=\{S_1,S_2,S_3,\cdots\}$  that has the property
$\mathcal{P}$ ($S$ could be finite in which case the sequence $\{S_i\}_i$ is
eventually constant) and fix a sentence $\psi$.
For any given {\sc re} theory $T$ we construct the theory
$T'=\{T'_1,T'_2,T'_3,\cdots\}$ as follows: let $T'_k=S_k$ if $k$ is not a
(code of a) proof of $\psi\wedge\neg\psi$ in $\widehat{T}$ (see above);
otherwise let $T'_k=\psi\wedge\neg\psi$.
Note that this construction is algorithmic, since being a
(code for a) proof  of $\psi\wedge\neg\psi$ in the decidable set $\widehat{T}$
is decidable (and the theory $\widehat{T}$ can be constructed algorithmically for a
given {\sc re} theory $T$); moreover, the theory $T'$
is {\sc re}. Now, if $T$ is consistent then $T'=S$ has the
property $\mathcal{P}$ and if $T$ is not consistent then $T'$ is an inconsistent
theory (because then for some $k$, $T'_k=\psi\wedge\neg\psi$) and so does not have the property $\mathcal{P}$. Whence, for any {\sc re}
theory $T$ we have the {\sc re} theory $T'$ in such a way that
\newline\centerline{$\big[T$ is consistent$\big]\iff\big[T'$ has the property $\mathcal{P}\big]$.} Now by
Corollary \ref{cor-con} the property $\mathcal{P}$ is not decidable.
\end{proof}

Finally, we note that as a corollary to the above theorem,
finite axiomatizability of {\sc re} theories is not a decidable property;
and there exists a decidable non--trivial property for finite theories:
for a fixed decidable theory (like the theory $\{A_1,\cdots,A_6\}$ in Definition~\ref{def-ttheory}), say $F$,
it is decidable whether a given finite theory $T$ is included in $F$
(i.e., if $F$ can prove all the  sentences of $T$).

%%%%%%%%%%%%%%%%%%%%%%%%%
%%%%%%%%%%%%%%%%%%%%%%%%%%%%%%%%%%%%
%%%%%%%%%%%%%%%%%%%%%%%%%

\section*{References}

\noindent   {\sc Beklemishev}, Lev D. \\
 (2010) \  G\"odel Incompleteness Theorems and the Limits of Their Applicability: I, {\it Russian Mathematical Surveys} {65}(5) 857--899.
  %\\  {\tt doi: 10.1070/RM2010v065n05ABEH004703}

\bigskip

\noindent {\sc  Berto}, Francesco\\
 (2009) \ {\it There's Something About G\"odel: the complete guide to the incompleteness theorem}, Wiley--Blackwell.

\bigskip

\noindent {\sc Boolos}, George S.  \&  {\sc Burgess}, John P. \&  {\sc Jeffrey}, Richard C.\\ (2007) \ {\it Computability and Logic}, Cambridge University Press   (5th ed).

\bigskip

\noindent  {\sc  Oliveira}, Igor Carboni \& {\sc Carnielli}, Walter \\  (2008) \ The Ricean Objection: an analogue of Rice's theorem for first--order theories,    {\it Logic Journal of the IGPL} {16}(6)   585--590.
%\\ {\tt doi: 10.1093/jigpal/jzn023}
\\
  (2009) \
 Erratum to ``The Ricean Objection: an analogue of Rice's theorem for first--order theories",    {\it Logic Journal of the IGPL} {17}(6)   803--804.
 %\\     {\tt doi: 10.1093/jigpal/jzp030}

\bigskip

\noindent  {\sc Chiswell}, Ian  \& {\sc Hodges}, Wilfrid \\ (2007) \ {\it Mathematical Logic}, Oxford University Press.

\bigskip

\noindent   {\sc Craig}, William \\  (1953) \ On Axiomatizability Within a System, \\ {\it The Journal of Symbolic Logic} 18(1)   30--32.

\bigskip

\noindent
 {\sc  Enderton}, Herbert B. \\ (2001) \ {\it A Mathematical Introduction to Logic}, Academic Press  (2nd ed).

\bigskip

\noindent {\sc  Epstein}, Richard L. \& {\sc Carnielli }, Walter A. \\ (2008)  \
{\it Computability: computable functions, logic, and the foundations of mathematics},
Advanced Reasoning Forum  (3rd ed).

\bigskip

\noindent  {\sc  Isaacson}, Daniel \\ (2011) \ Necessary and Sufficient Conditions for Undecidability of the G\"odel Sentence and its Truth,
 in:   {\it Logic, Mathematics, Philosophy, Vintage Enthusiasms: essays in honour of John L. Bell}, D. DeVidi \&  M. Hallett \&  P. Clark (Eds.)
 Springer, 135--152. {\small {\sf doi: 10.1007/978-94-007-0214-1$_-$7}}
% \\   {\tt doi: 10.1007/978-94-007-0214-1_7}

\bigskip

\noindent   {\sc  Kaye}, Richard W. \\ (2007) \ {\it The Mathematics of Logic: a guide to completeness theorems and their applications}, Cambridge University Press.

\bigskip

\noindent   {\sc Kleene},  Stephen C. \\  (1936) \ General Recursive Functions of Natural Numbers,
 \\ {\it Mathematische Annalen} {112}(1)   727--742.
%  \\   {\tt doi: 10.1007/BF01565439}
\\ (1950) \  A Symmetric Form of G\"odel's Theorem, \\
  {\it Indagationes Mathematicae} {12},   244--246. \\
(1952) \  {\it Introduction to Metamathematics}, North-Holland.

\bigskip

\noindent   {\sc  Lafitte }, Gr\'egory \\ (2009) \
Busy Beavers Gone Wild,
 in:  {\it  Proceedings of the International Workshop on The Complexity of Simple Programs
 (CSP, Cork, Ireland, 6--7 December 2008)}, T. Neary \& D. Woods \& T. Seda \& N. Murphy (Eds.),
  Electronic Proceedings in Theoretical Computer Science~{1},
 Open Publishing Association, 123--129.  	 {\small {\sf doi: 10.4204/EPTCS.1.12}}
%\\  \url{http://rvg.web.cse.unsw.edu.au/eptcs/paper.cgi?CSP08.12}
%\\   {\tt doi: 10.4204/EPTCS.1.12}

\bigskip

\noindent  {\sc  Li}, Ming \&  {\sc Vit\'anyi}, Paul M.B. \\ (2008) \
{\it An Introduction to Kolmogorov Complexity and Its Applications},  \\
Springer (3rd ed).

\bigskip

\noindent  {\sc  Rosser}, Barkley \\ (1936) \ Extensions of Some Theorems of G\"odel and Church, \\
 {\it The Journal of Symbolic Logic} {1}(3)   87--91.
%\\    \url{http://www.jstor.org/stable/2269028}

\bigskip

\noindent  {\sc  Tarski}, Alfred (with Mostowski, Andrzej  \&  Robinson, Raphael M.) \\  (1953) \ {\it Undecidable Theories}, \\ North--Holland; reprinted by Dover Publications 2010.

\bigskip

\noindent
 {\sc  van Dalen}, Dirk \\  (2013) \ {\it Logic and Structure}, Springer  (5th ed).

\bigskip

\noindent  {\sc  Wasserman}, Wayne U. \\ (2008) \
It is `Pulling a Rabbit Out of the Hat': typical diagonal lemma `proofs' beg the question,
  {\it Social Science Research Network}. {\small {\sf doi: 10.2139/ssrn.1129038}}
%\\   \url{http://dx.doi.org/10.2139/ssrn.1129038}

\bigskip

\vfill

{\small

{\sc Saeed Salehi}

\quad Department of Mathematical Sciences, University of Tabriz,

 \qquad 29 Bahman Boulevard, P.O.Box~51666--17766, Tabriz, IRAN.

 \quad School of Mathematics, Institute for Research in Fundamental  Sciences (IPM),

 \qquad Niavaran, P.O. Box 19395--5746, Tehran, IRAN.

  \quad%\qquad
  E-mail: \ {\tt salehipour@tabrizu.ac.ir}

  \quad%\qquad
  URL: \ {\tt http://saeedsalehi.ir/}

}

\end{document}